\documentclass[article]{siamart}
\usepackage{amsmath,amssymb}
\usepackage{caption}
\usepackage{hyperref}
\usepackage[utf8]{inputenc}
\usepackage{mathtools}
\usepackage{color,graphicx}
\usepackage{enumitem}
\usepackage{tikz}
\usetikzlibrary{patterns}
\usetikzlibrary{arrows}
\usetikzlibrary{3d,calc}

\makeatletter

\newcommand{\cO}{\mathcal{O}}

\newcommand{\beq}{\begin{equation}}
\newcommand{\eeq}{\end{equation}}
\def\bals#1\eals{\begin{align*} #1 \end{align*}}
\def\bal#1\eal{\begin{align} #1 \end{align}}

\newcommand\Dom\Omega

\newcommand\ZZ{\mathbb{Z}}

\newcommand\Lap\Delta

\def\bpde#1\epde{\[\left\{\begin{aligned}#1\end{aligned}\right. \]}
\def\inbpde#1\inepde{\left\{\begin{aligned}#1\end{aligned}\right.}
\def\binpde#1\einpde{\left\{\begin{aligned}#1\end{aligned}\right.}

\def\b0{\mathbf{0}}

\def\bbmat{\begin{bmatrix}[r]}
\def\ebmat{\end{bmatrix}}

\newcommand{\barr}{\begin{array}}
\newcommand{\ea}{\end{array}}
\newcommand{\bea}{\begin{eqnarray}}
\newcommand{\eea}{\end{eqnarray}}
\newcommand{\bt}{\begin{table}}
\newcommand{\et}{\end{table}}

\theoremstyle{plain}
\theoremstyle{definition}

\newsiamthm{cond}{Condition}
\newsiamremark{remark}{Remark}

\numberwithin{equation}{section}

\newcommand{\TheTitle}{Exact and Fast Inversion of the Approximate Discrete
Radon Transform}

\begin{document}

\title{\TheTitle}

\author{Donsub Rim%
  \thanks{Courant Institute of Mathematical Sciences, %
        New York University, 251 Mercer St., New York, NY 10012 %
  (\email{{\tt dr1653@nyu.edu}}).}%
}
\maketitle

\begin{abstract} 
We give an exact inversion formula for the approximate discrete Radon transform
introduced in [Brady, \emph{SIAM J. Comput}., 27(1), 107--119] that is
of cost $\cO(N \log N)$ for a square 2D image with $N$ pixels.
\end{abstract}

\section{Introduction}

The Radon transform is a linear transform that integrates a real-valued
function on a $d$-dimensional Euclidean space over all possible
$(d-1)$-dimensional hyperplanes \cite{helgason}. The transform is a widely used
model in tomography, and many inverse problems involving the transform or its
variant have been carefully studied \cite{natterer}. 

We will refer to the Radon transform as the continuous Radon transform, to
distinguish it from the approximate discrete Radon transform (ADRT), which
is the focus of this paper; a discrete version of the transform that
replaces the smooth hyperplanes by broken pixelated lines, called \emph{digital
lines} \cite{brady,GD96}. ADRT is computed by summing the pixel values that lie
in these lines. The digital lines are computed recursively, yielding a fast
transform of computational cost $\cO(N \log N)$ for a square 2D image with $N$
number of pixels.

We will show that the exact inverse can be computed with the same complexity,
in contrast to other discrete versions of the continuous transform. This is
somewhat surprising: despite the known inversion formula for the continuous
case, computing the inverse requires more effort than the forward transform,
and the digital lines in ADRT converges to straight lines upon repeated
refinement. Moreover, the inverse can be computed using only partial data, that
is, one quadrant of the ADRT that correspond to angles in $[0,\pi/4]$ for the
continuous transform.

It was observed that the inverse of ADRT can be computed to numerical precision
by employing a multigrid method \cite{pressdrt}, and this was further validated
using the conjugate gradient method \cite{radonsplit}. However, the simple
formula in this work is new, to the best of our knowledge.

\section{Main Result}

Our main assertion is that the ADRT is exactly invertible from partial
data, with the same computational complexity as the forward transform.

\begin{theorem}\label{thm}
A 2D square image with $N$ pixels can be computed exactly from its
single-quadrant ADRT in $\cO(N \log N)$ operations.  
\end{theorem}

Let us be given a 2D square image $A$ containing $N$ pixels taking on real
values, with the dimensions $2^n \times 2^n$. We will first define the ADRT for
a rectangular sub-images of size $2^n \times 2^m$ for $m < n$, called sections.

\begin{definition}[Section of an image]
Let the \emph{$\ell$-th section of a $2^n \times 2^n$ image $A$} be defined on
$\ZZ^2$ given by
\beq
    A^\ell_{n,m-1}(i,j)
    :=
    \begin{cases}
       A(i,j + (\ell-1)2^{m-1}) &\text{ if } i = 1 , \cdots, 2^n, %
                                         j =1 , \cdots, 2^{m-1},\\
       0 & \text{ otherwise,}
    \end{cases}
\eeq
where $\ell = 1,2, \cdots, 2^{n-m}$.
\end{definition}

Next, we define the broken lines used in ADRT.

\begin{definition}[Digital line]\label{def:dline}
A \emph{digital line} $D_{m} (h,s)$, for $h \in \ZZ$, $s = 1, \cdots, 2^m$ is a
subset of $\ZZ^2$ that is defined recursively. Letting $s = 2t$ or $s = 2t+1$,
\beq
    \left\{
    \begin{aligned}
    D_{m} (h, 2t+1) &:= D^1_{m-1}(h,t) \cup D^2_{m-1}(h+t+1,s),\\
    D_{m} (h, 2t)   &:= D^1_{m-1}(h,t) \cup D^2_{m-1}(h+t,s),
    \end{aligned}
    \right.
    \label{eq:dline}
\eeq
for $t=1, \cdots, 2^{m-1}$ and where
\beq
    D^1_m(h,s) := D_m(h,s),
    \quad
    D^2_m(h,s) := \{(i,j+2^m) : (i,j) \in D_m(h,s) \},
\eeq
and the relation is initialized by $D_0 (h,s) := \{(h,1): h \in \ZZ\}$.
\end{definition}

Then ADRT is a sum of pixels that lie on the digital lines.

\begin{figure}
\centering
\begin{tabular}{ccc}
\begin{tikzpicture}[scale=0.36]
\draw[step=1,gray,very thin] (-1,-1) grid (9,9);
\draw[fill=black] (0,0) -- (2,0) -- (2,1) -- (0,1) -- (0,0);
\draw[fill=black] (2,1) -- (4,1) -- (4,2) -- (2,2) -- (2,1);
\draw[fill=black] (4,2) -- (6,2) -- (6,3) -- (4,3) -- (4,2);
\draw[fill=black] (6,3) -- (8,3) -- (8,4) -- (6,4) -- (6,3);
\draw[black,thick,dashed] (0,-1) -- (0,9);
\draw[black,thick,dashed] (8,-1) -- (8,9);
\draw[black] (0,0.5) node[anchor=east] {$h=1$};
\draw[black] (2.5,0.5) node[anchor=west] {$s=3$};
\draw[fill=blue] (0,4) -- (2,4) -- (2,5) -- (0,5) -- (0,4);
\draw[fill=blue] (2,5) -- (6,5) -- (6,6) -- (2,6) -- (2,5);
\draw[fill=blue] (6,6) -- (8,6) -- (8,7) -- (6,7) -- (6,6);
\draw[blue] (0,4.5) node[anchor=east] {$h=4$};
\draw[blue] (2.5,6.5) node[anchor=west] {$s=2$};
\end{tikzpicture}
&
\medmuskip=0.0\medmuskip
\thickmuskip=0.0\thickmuskip
\begin{tikzpicture}[scale=0.9]
\draw (-2,0) -- (-0.1,0) -- (-0.1,4) -- (-2,4) -- (-2,0); 
\draw (2,0) -- (0.1,0) -- (0.1,4) -- (2,4) -- (2,0); 
\draw[blue,very thick] (-2,0.5) 
           -- (-0.1,2.0) node[sloped,midway,below] {\scriptsize $D^1_m(h,s)$};
\draw (-0.1,2.0) node[anchor=east,blue] {\small $h+s$} ;
\draw[blue,thin] (-0.1,2.0) -- (0.1,2.0);
\draw[blue,very thick] (2,3.5) 
           -- (0.1,2.0) node[sloped,midway,above] {\scriptsize $D^2_m(h+s,s)$};
\draw (0.1,2.0) node[anchor=west,blue] {\small $h+s$};
\end{tikzpicture} 
&
\medmuskip=0.1\medmuskip
\thickmuskip=0.125\thickmuskip
\begin{tikzpicture}[scale=0.9]
\draw (-2,0) -- (-0.1,0) -- (-0.1,4) -- (-2,4) -- (-2,0); 
\draw (2,0) -- (0.1,0) -- (0.1,4) -- (2,4) -- (2,0); 
\draw[black,very thick] (-2,0.5) 
    -- (-0.1,2.1) node[sloped,midway,below]{\scriptsize $D^1_m(h,s)$};  
\draw (-0.1,2.1) node[anchor=east] {\small $h+s$} ;
\draw[black,very thick] (2,3.8) 
        -- (0.1,2.27) node[sloped,midway,above]{\scriptsize $D^2_m(h+s+1,s)$};
\draw (0.1,2.27) node[anchor=west] {\small $h+s+1$};
\draw[black,thin] (0.1,2.27) -- (-0.1,2.1);
\end{tikzpicture} 
\end{tabular}
\caption{Examples of digital lines (d-lines) $D_{m}(h,s)$ where $m=3$ (left)
and a diagram illustrating the recursion definition \cref{eq:dline} (middle,
right). In both figures the case when $s$ is even is depicted in blue, odd in
black, respectively.} 
\label{fig:dlines}
\end{figure}
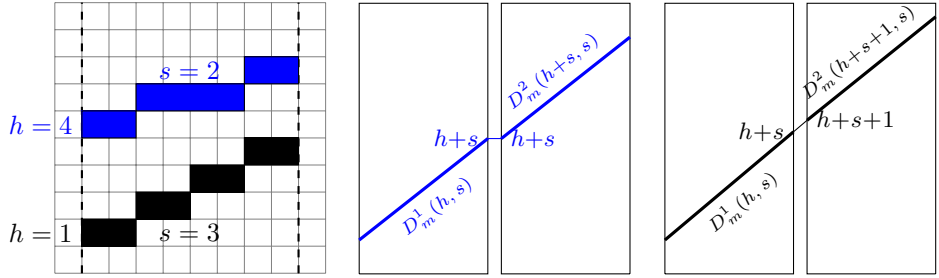

\begin{definition}[Single-quadrant ADRT]
We will denote the \emph{single-quadrant approximate discrete Radon transform
(ADRT) applied to the $\ell$-th section of the image $A$} by $R^\ell_{n,m}$,
defined to be the sum
\beq
    R^\ell_{n,m}(h,s) := \sum_{(i,j) \in D_m(h,s)} A^\ell_{n,m}(i,j),
    \quad \text{ where } \ell = 1, \cdots, 2^{n-m}.
\eeq  
In particular, when $m=n$ we call the sums simply the \emph{single-quadrant
ADRT of $A$}, and we denote it by $R_m := R^1_{n,n}$.
\end{definition}

A full ADRT is obtained when one computes a single-quadrant ADRT four times
upon flipping and rotating the image \cite{brady}. 

We state some basic properties of digital lines and the ADRT, omitting the
proof.

\begin{lemma}[Properties]\label{lem:prpt} ~
\begin{enumerate}[label=(\roman*)]
    \item $\{D_m(h,s): h \in \ZZ\}$ form a partition of 
    the set $\{(i,j): i \in \ZZ, j  = 1, \cdots, 2^m\}$.
    \item $D_m^1(h,s) \cap D_m^2(k,t) = \emptyset$,
    for all $h,k \in \ZZ$ and $s,t = 1, \cdots, 2^m$.
    \item 
$D_m(h,s) \subset \{(i,j) \in \ZZ^2: h \le i \le h+s, 1 \le j \le 2^m\}$.
    \item $R^\ell_{n,0}(h,1) = A(h,\ell)$ for $h \in \ZZ, \ell = 1, \cdots, 2^m$.
    \item $R^\ell_{n,m}(h,s) = 0$ when $h < -s+1$ or $h > 2^n$, 
          $s = 1, \cdots, 2^m$. 
\end{enumerate}
\end{lemma}

The key lemma below states that the recursive definition in ADRT is reversible:
the ADRT of a section of an image can be computed from that of the whole. 

Note that by \cref{lem:prpt}, $R^\ell_{n,m}$ can have at most $2^n \cdot 2^m +
2^{m-1}(2^{m} + 1)$ non-zero entries.

\begin{lemma}\label{lem:localize}
Given $R^\ell_{n,m}$, $\{R^{2\ell-1}_{n,m-1},R^{2\ell}_{n,m-1}\}$ can be
computed in $\cO(M)$ operations, where $M := 2^{m-1}(2^{n+1} + 2^m +1)$ is the
size of $R^\ell_{n,m}$.
\end{lemma}

\begin{proof}
Let us define the differences
\beq
   \Delta^\ell_{n,m-1} (h,s) := R^\ell_{n,m-1}(h+1,s)
                              - R^\ell_{n,m-1}(h,s),
\eeq
then these differences can be computed by the relations
\begin{align}
    \Delta^{2\ell-1}_{n,m-1} (h,s) &= R^\ell_{n,m}(h+1,2s)
                                  - R^\ell_{n,m}(h,2s+1),
    \label{eq:diff_even}\\
    \Delta^{2\ell}_{n,m-1} (h,s) &= R^\ell_{n,m}(h-s,2s+1)
                                    - R^\ell_{n,m}(h-s,2s).
    \label{eq:diff_odd}
\end{align}
They follow from direct computation. For example,
\[
    \begin{aligned}
      R^\ell_{n,m} (h+1, 2s) &=
            \sum_{(i,j) \in D_m(h+1,s)} A^\ell_{n,m} (i,j) \\
            &=
            \sum_{(i,j) \in D^1_{m-1}(h+1,s)} A^{2\ell-1}_{n,m-1} (i,j)
            +
            \sum_{(i,j) \in D^2_{m-1}(h+s+1,s)} A^{2\ell}_{n,m-1} (i,j) \\
            &=
            R^{2\ell-1}_{n,m-1}(h+1,s)
            +
            \sum_{(i,j) \in D_{m-1}(h,s)}A^{2\ell}_{n,m-1}(i+s+1,j),\\
    \end{aligned}
\]
and also
\[
    \begin{aligned}
        R^\ell_{n,m} (h, 2s+1) &=
            \sum_{(i,j) \in D_{m}(h,s)} A^\ell_{n,m} (i,j) \\
            &=
            \sum_{(i,j) \in D^1_{m-1}(h,s)} A^{2\ell-1}_{n,m-1} (i,j)
            +
            \sum_{(i,j) \in D^2_{m-1}(h+s+1,s)} A^{2\ell}_{n,m-1} (i,j) \\
            &=
            R^{2\ell-1}_{n,m-1}(h,s)
            +
            \sum_{(i,j) \in D_{m-1}(h,s)} A^{2\ell}_{n,m-1}(i+s+1,j),\\
    \end{aligned}
\]
so we have \cref{eq:diff_even} upon subtraction. Since we now have the
differences, 
\beq
 R^{\ell}_{n,m-1} (h,s) = \sum_{k=-\infty}^{h-1} \Delta^{\ell}_{n,m-1} (k,s)
                           = \sum_{k=-s}^{h-1} \Delta^{\ell}_{n,m-1}(k,s),
\eeq
since part of the sum vanishes due to \cref{lem:prpt}.

For each fixed $s$, this sum allows one to compute all values of $h$ in one
sweep with one addition per one $h$: so the number of additions or subtractions
for computing $R^\ell_{n,m-1}$ for all $(h,s)$ is $2M = 2\cdot 2^{m-1}(2^{n+1}
+ 2^m + 1) = \cO(N)$.
\end{proof}

The main theorem follows.

\begin{proof}[\cref{thm}]
By \cref{lem:localize}, we compute  $\{R^\ell_{n,m-1}: \ell = 1, \cdots ,
2^{n-m+1}\}$ from $\{R^\ell_{n,m}: \ell = 1, \cdots, 2^{n-m}\}$. Repeating this
procedure for $m = 1, \cdots, n$, one computes $R^\ell_{n,0}$ and by doing so
one obtains the original image $A$. 

The cost is then
\begin{equation}
    \sum_{m=1}^n 2^{n-m} \cdot [2 \cdot 2^{m-1}(2^{n+1} + 2^m + 1)]
    = (2N + \sqrt{N}) \log N + 4N - 2\sqrt{N},
\end{equation}
that is, $\cO(N \log N)$.
\end{proof}

\section{Conclusion}
We showed that the inverse of ADRT can be computed from $\cO(N \log N)$
operations, using data from only one quadrant of the ADRT. This could be useful
for applications in computational applications of the ADRT, e.g. for use in
dimensional splitting or in approximating the continuous transform. It also has
implications for implementation, especially when decomposing the domain for
parallelization. This result indicates that the relation between the ADRT and
the continuous transform is more involved than it may first appear.

\bibliographystyle{siamplain}

\end{document}